\documentclass[11pt,a4paper]{article}
\usepackage[T1]{fontenc}
\usepackage{lmodern}
\usepackage{amsmath,amssymb,amsthm,mathtools}
\usepackage[left=27mm,right=27mm,top=27mm,bottom=25mm]{geometry}
\usepackage{microtype}
\usepackage{fancyhdr}
\usepackage[hidelinks,hyperfootnotes=false]{hyperref}
\allowdisplaybreaks
\numberwithin{equation}{section}
\newtheorem{theorem}{Theorem}[section]
\newtheorem{lemma}[theorem]{Lemma}
\newtheorem{proposition}[theorem]{Proposition}
\newtheorem{corollary}[theorem]{Corollary}
\theoremstyle{definition}
\newtheorem{definition}[theorem]{Definition}
\theoremstyle{remark}
\newtheorem{remark}[theorem]{Remark}
\newcommand{\vp}{\dot\partial}
\newcommand{\C}{\mathbb C}
\newcommand{\R}{\mathbb R}
\newcommand{\B}{\mathbb B}
\newcommand{\D}{\mathbb D}
\newcommand{\tr}{\operatorname{tr}}
\newcommand{\id}{\operatorname{id}}
\newcommand{\dd}{\mathrm d}
\newcommand{\Mt}{\widetilde M}
\title{Rigidity of K\"{a}hler--Finsler Metrics\\with Constant Radial Curvature}
\author{Jinling Li, Longjie Li, and Chunhui Qiu}
\date{}
\begin{document}
\maketitle
\thispagestyle{plain}
\begin{abstract}
In this paper, we prove that a smooth strongly pseudoconvex
K\"ahler--Finsler metric $F$ satisfying the constant  radial curvature condition 
\[
\Omega(\chi,\bar\chi)\chi=cF^2\chi,
\qquad c\in\mathbb R\setminus\{0\},
\]
is induced by a Hermitian K\"ahler metric of holomorphic sectional
curvature $2c$. Under the Abate--Patrizio curvature symmetry
\[
\langle\Omega(\cdot,\bar\chi)\chi,\chi\rangle
=\langle\Omega(\chi,\bar\chi)\cdot,\chi\rangle,
\]
this result gives Hermitian rigidity for nonzero constant holomorphic
sectional curvature and the classification of complete simply connected
metrics of negative curvature. As an application, we show that the
Kobayashi metric of a bounded strongly convex domain with smooth
boundary is K\"ahler--Finsler if and only if the domain is biholomorphic
to the unit ball. We also prove that the universal cover of a bounded
strongly pseudoconvex domain $D $ with smooth
boundary is biholomorphic to the unit ball if and only if its Kobayashi
metric is a smooth strongly pseudoconvex K\"ahler--Finsler metric.
Furthermore, we give two explicit examples to illustrate the distinction between the weakly
K\"ahler--Finsler and K\"ahler--Finsler conditions.

\end{abstract}
\begingroup
\renewcommand{\thefootnote}{}
\footnotetext{This work is Supported by the Natural Science Foundation of
Fujian Province (Grant No.\ 2025J01350) and the National Natural Science
Foundation of China (Grant Nos.\ 12431005, 12371081).}
\endgroup
\section{Introduction}
The classification of complete simply connected K\"ahler manifolds of
constant holomorphic sectional curvature is a classical result in
complex geometry (\cite{KN}). In complex Finsler geometry,
Abate and Patrizio \cite{AP97} obtained classification results for
K\"ahler--Finsler metrics of nonnegative constant holomorphic curvature
under an additional curvature symmetry. 
In this paper, we consider smooth strongly pseudoconvex
K\"ahler--Finsler metrics satisfying the radial curvature identity
\eqref{eq:1.1} with $c\ne0$. We prove that every such metric is induced
by a Hermitian K\"ahler metric of holomorphic sectional curvature $2c$.
This rigidity result is local and applies to both positive and negative
$c$. In the complete simply connected case, it yields the corresponding
classification by the classical K\"ahler space forms.

Write \(G=F^2\), let \(\Omega\) be the curvature operator of the Chern--Finsler
connection, and let \(\chi=v^j\delta_j\) be its horizontal radial field.
Our first main result is as follows.

\begin{theorem}\label{thm:main}
Let \(F\) be a smooth strongly pseudoconvex K\"ahler--Finsler metric on a complex
manifold \(M\). If
\begin{equation}\label{eq:1.1}
\Omega(\chi,\bar\chi)\chi=cG\chi
\quad\text{on }T^{1,0}M\setminus\{0\},\qquad c\in\R\setminus\{0\},
\end{equation}
then \(F\) is induced by a Hermitian K\"ahler metric of holomorphic sectional
curvature \(2c\).
\end{theorem}

For a K\"ahler--Finsler metric, Abate and Patrizio \cite{AP97}
showed that \eqref{eq:1.1} is equivalent to the condition $K_F=2c$
together with the curvature symmetry
\begin{equation}\label{eq:1.2}
\langle\Omega(H,\bar\chi)\chi,\chi\rangle_g
=\langle\Omega(\chi,\bar\chi)H,\chi\rangle_g
\qquad(H\in\mathcal H),
\end{equation}
where $\mathcal H$ denotes the horizontal bundle. For \(c>0\), Hermitian rigidity follows from the local argument
in the proof of Theorem~1.1  in  \cite{AP97}.
Aikou \cite{Ai95} obtained related rigidity results for complex
manifolds modeled on a complex Minkowski space.
As a consequence of Theorem~\ref{thm:main}, we obtain the following
classification.
\begin{theorem}\label{thm:classification}
Let $F$ be a smooth strongly pseudoconvex K\"ahler--Finsler metric on a
connected complex manifold $M$ of dimension $n$ with $K_F=2c\ne0$,
satisfying \eqref{eq:1.2}. Then $F$ is induced by a Hermitian
K\"ahler metric. If, in addition, $c<0$, $M$ is simply connected,
and  $F$ is complete, then
\[
(M,F)\cong
\left(\B^n,\sqrt{-\frac2c}\,\kappa_{\B^n}\right)
\]
by a biholomorphic isometry. 
\end{theorem}
Our first application concerns bounded strongly convex domains with
smooth boundary. By Lempert's theory  (\cite{Lempert81,Lempert82}), their
Kobayashi metrics are smooth and strongly pseudoconvex. They are weakly
K\"ahler--Finsler and satisfy \eqref{eq:1.1} with $c=-2$
(\cite{AP,AP96}). We give a direct proof using holomorphic retracts.
Ge \cite{Ge} showed that, for \(t>0\), the Kobayashi metric of
\[
\Omega_t
=\{(z_1,z_2)\in\C^2:
|z_1|^2+|z_2|^2+t|z_1|^4<1\}
\]
is smooth and weakly K\"ahler--Finsler, but is not
K\"ahler--Finsler. 
\begin{theorem}\label{thm:convex}
Let \(\kappa_D\) be the Kobayashi metric on a bounded strongly convex domain
\(D\Subset\C^n\) with \(C^\infty\) boundary. Then \(\kappa_D\) is a smooth strongly pseudoconvex weakly
K\"ahler--Finsler metric and satisfies \eqref{eq:1.1}
with \(c=-2\).  The following conditions are equivalent:
\begin{enumerate}
\item[(i)] \(\kappa_D\) is a K\"ahler--Finsler metric;
\item[(ii)] \(\kappa_D\) is a Hermitian metric;
\item[(iii)] \(D\) is biholomorphic to \(\B^n\).
\end{enumerate}
\end{theorem}

 For bounded strongly pseudoconvex domains,  Gaussier and Zimmer \cite{GZ} proved that
the Kobayashi metric is Hermitian K\"ahler if and only if the universal
cover is biholomorphic to the unit ball. A version assuming the Hermitian
K\"ahler condition near the boundary is noted in
\cite{DWW}. The next theorem replaces that
assumption by the smooth strongly pseudoconvex K\"ahler--Finsler condition.

\begin{theorem}\label{thm:boundary}
Let \(D\Subset\C^n\), \(n\ge2\), be strongly pseudoconvex with \(C^\infty\)
boundary, and let \(\kappa_D\) be the Kobayashi metric on $D$.
Then the following conditions are equivalent:
\begin{enumerate}
\item[(i)] \(\kappa_D\) is a smooth strongly pseudoconvex K\"ahler--Finsler
metric on \(D\);
\item[(ii)] There exists \(\epsilon>0\) such that the restriction
of \(\kappa_D\) to
\[
V_\epsilon
=\{z\in D:\operatorname{dist}(z,\partial D)<\epsilon\}
\]
is a smooth strongly pseudoconvex K\"ahler--Finsler metric;
\item[(iii)] \(\kappa_D\) is a Hermitian K\"ahler metric of holomorphic sectional
curvature \(-4\) on \(D\);
\item[(iv)] the universal cover \(\widehat D\) is biholomorphic to \(\B^n\).
\end{enumerate}
Under these conditions, \(D\simeq\B^n\) if \(D\) is simply connected.
\end{theorem}
The equivalence of \textup{(iii)} and \textup{(iv)} in
Theorem~\ref{thm:boundary} follows from  Theorem~1.5   in 
\cite{GZ}.
The smoothness and strong pseudoconvexity in \textup{(i)} and
\textup{(ii)} are explicit assumptions.
In \textup{(ii)}, we use the restriction of \(\kappa_D\)
to \(V_\epsilon\).

The paper is organized as follows. In Section~2, we introduce the basic
definitions and notation of complex Finsler geometry. In Section~3, we
establish several identities for K\"ahler--Finsler metrics and
prove the classification theorem. In Section~4, we study the
K\"ahler--Finsler characterization of the unit ball. In Section~5, we
address the corresponding characterization of ball universal
covers for strongly pseudoconvex domains. In Section~6, we give two
examples of smooth weakly K\"ahler--Finsler metrics that are not
K\"ahler--Finsler.

\section{Preliminaries and conventions}
\subsection{The nonlinear connections and Chern-Finsler connections}

Let \(M\) be a complex manifold of dimension \(n\), and let
\(\pi:T^{1,0}M\to M\) be the holomorphic tangent bundle of \(M\).
Let \(z=(z^1,\dots,z^n)\) be local holomorphic coordinates on \(M\),
and let \((z,v)\) be the induced coordinates on \(T^{1,0}M\), where
\[
v=v^i\frac{\partial}{\partial z^i}.
\]

All smoothness statements refer to the complement of the zero section.
We set \(\Mt=T^{1,0}M\setminus\{0\}\).
A strongly pseudoconvex complex Finsler metric $F$ is a non-negative continuous
function on \(T^{1,0}M\) such that \(G=F^2\) is a smooth positive function on \(\Mt\),
\[
G(z,\lambda v)=|\lambda|^2G(z,v)\qquad((z,v)\in\Mt,\lambda\in\C^*),
\]
and the fundamental tensor
\[
g_{i\bar j}=\frac{\partial^2G}{\partial v^i\partial\bar v^j}
\]
is positive definite on \(\Mt\). Note that $G$ is smooth on the whole \(T^{1,0}M\) iff $F$ comes from a Hermitian metric. We use the notation
\[
\partial_j=\frac{\partial}{\partial z^j},\quad
\vp_j=\frac{\partial}{\partial v^j},\quad G_i=\vp_iG,\quad G_{;j}=\partial_jG.
\]
A bar on an index denotes differentiation in the conjugate variable, and
a semicolon separates the base from fiber derivatives. Thus \(\partial_{\bar j}\)
differentiates in \(\bar z^j\), whereas \(\vp_{\bar j}\) differentiates in
\(\bar v^j\). In particular, set
\[
A_{ij}=G_{ij},\qquad g_{i\bar j}=G_{i\bar j},\qquad P_{j\bar k}=G_{;j\bar k}.
\]
The matrix \(A\) is symmetric; it is not the Hermitian fundamental tensor \(g\).

Euler's identities give (Abate--Patrizio \cite{AP})
\begin{align}
v^iG_i&=G,\qquad \bar v^jG_{\bar j}=G,\label{eq:2.1}\\
g_{i\bar j}\bar v^j&=G_i,\qquad g_{i\bar j}v^i=G_{\bar j},
\label{eq:2.2}\\
A_{ij}v^j&=0,\qquad v^i\vp_rg_{i\bar j}=0.\label{eq:2.3}
\end{align}
We use the Einstein summation convention. The inverse tensor is denoted by
\(g^{\bar j i}\), so that \(g_{k\bar j}g^{\bar j i}=\delta^i_k\).

The nonlinear connection and horizontal derivatives are
\begin{equation}
N^i_j=g^{\bar a i}G_{\bar a;j},\quad
\delta_j=\partial_j-N^r_j\vp_r,\quad
\delta_{\bar j}=\partial_{\bar j}-\overline{N^r_j}\vp_{\bar r}.
\label{eq:2.4}
\end{equation}
The horizontal and vertical connection coefficients of Chern-Finsler
connection are
\begin{equation}
\Gamma^i_{jk}=g^{\bar a i}\delta_jg_{k\bar a},\qquad
C^i_{kr}=g^{\bar a i}\vp_rg_{k\bar a}.
\label{eq:2.5}
\end{equation}

\begin{lemma}[{\cite{AP}}]\label{lem:basic}
For every strongly pseudoconvex complex Finsler metric, we have
\begin{align}
\Gamma^i_{jk}&=\vp_kN^i_j,\qquad \Gamma^i_{jk}v^k=N^i_j,
\label{eq:2.6}\\
G_iN^i_j&=G_{;j},\qquad \delta_jG=\delta_{\bar j}G=0,
\label{eq:2.7}\\
\delta_jG_{\bar i}&=0,\qquad \delta_{\bar j}G_i=0,
\label{eq:2.8}\\
\delta_jG_k&=G_i\Gamma^i_{jk}.\label{eq:2.9} 
\end{align}
\end{lemma}
\begin{proof}
Differentiating the inverse matrix in \eqref{eq:2.4}, and using the symmetry
of third fiber derivatives, we obtain
\[
\vp_kN^i_j=g^{\bar a i}(G_{k\bar a;j}-N^r_jG_{k\bar a r})=\Gamma^i_{jk}.
\]
Its contraction with \(v^k\) is \(N^i_j\) by homogeneity. By \eqref{eq:2.2},
\[
G_iN^i_j=\bar v^aG_{\bar a;j}=G_{;j}.
\]
This proves \eqref{eq:2.7}. By the definition of \(N\), we obtain
\(G_{\bar i;j}=g_{r\bar i}N^r_j\), which proves the first identity in
\eqref{eq:2.8}; the second is proved by the conjugation. Finally, by differentiating
\(G_k=g_{k\bar a}\bar v^a\) horizontally, \eqref{eq:2.9} is obtained.
\end{proof}

\subsection{Three torsion conditions and Radial curvature condition}

\begin{definition}
A strongly pseudoconvex complex Finsler metric $F$ is called 
 \emph{strongly K\"ahler--Finsler} if \(T^i_{jk}=0\),
\emph{K\"ahler--Finsler} if \(v^jT^i_{jk}=0\), and
\emph{weakly K\"ahler--Finsler} if
\begin{equation}
G_iT^i_{jk}v^j=0\quad\text{for every }k,\label{eq:2.11}
\end{equation}
where \(T^i_{jk}=\Gamma^i_{jk}-\Gamma^i_{kj}\).
\end{definition}
The expression in \eqref{eq:2.11} is coordinate invariant. For a Hermitian
metric, all three conditions reduce to the ordinary K\"ahler condition.

Define the spray, without a factor of \(1/2\), by
\begin{equation}
S^i=N^i_jv^j,\qquad B^i_{\bar\ell}=\vp_{\bar\ell}S^i.
\label{eq:2.10}
\end{equation}
\begin{lemma}[ {\cite{CS}}]\label{lem:strong}
The K\"ahler--Finsler and strongly K\"ahler--Finsler conditions are equivalent.
Under either condition,
\begin{equation}
\vp_kS^i=2N^i_k.\label{eq:2.12}
\end{equation}
\end{lemma}
\begin{proof}
By \eqref{eq:2.6} and homogeneity, we have
\[
v^jT^i_{jk}=v^j\vp_kN^i_j-v^j\vp_jN^i_k=\vp_kS^i-2N^i_k.
\]
If this expression vanishes, then
\(\Gamma^i_{jk}=\frac12\vp_k\vp_jS^i\), which is symmetric in \(j,k\).
The converse is immediate.
\end{proof}

\begin{lemma}\label{lem:weak}
The weak K\"ahler condition is equivalent to
\begin{equation}
v^jG_{k;j}-G_{;k}-S^rA_{kr}=0\qquad(k=1,\dots,n).
\label{eq:2.13}
\end{equation}
\end{lemma}
\begin{proof}
By \eqref{eq:2.5} and \eqref{eq:2.9}, we have
\begin{equation}
G_iT^i_{jk}v^j=v^j[\delta_j(G_k)-\delta_k(G_j)].
\label{eq:2.14}
\end{equation}
The first term is \(v^j\delta_j(G_k)=v^jG_{k;j}-S^rA_{kr}\).
Because \(v^jA_{jr}=0\), the second term \(v^j\delta_k(G_j)=G_{;k}\).
This proves the claim.
\end{proof}

Define
\begin{equation}
R^i_{j\bar k}=-\delta_{\bar k}N^i_j,\quad
R^i=R^i_{j\bar k}v^j\bar v^k,\quad r_{j\bar k}=G_iR^i_{j\bar k}.
\label{eq:2.15}
\end{equation}
These quantities are tensorial on the slit tangent bundle \(\Mt\). One can also see
this from the full horizontal curvature of the Chern-Finsler connection:
\begin{equation}
R^i_{k;j\bar\ell}=-\delta_{\bar\ell}\Gamma^i_{jk}
-C^i_{ks}\delta_{\bar\ell}N^s_j,\qquad
R^i_{k;j\bar\ell}v^k=R^i_{j\bar\ell}.
\label{eq:2.16}
\end{equation}
The contraction follows from \(C^i_{ks}v^k=0\) and \eqref{eq:2.6}.
Let \(\mathcal H\) be the horizontal bundle spanned locally by
\(\delta_1,\dots,\delta_n\), with the Hermitian metric induced by \(g\).
For \(\chi=v^j\delta_j\), we have
\[
\Omega(\chi,\bar\chi)\chi=R^i\delta_i.
\]
The holomorphic sectional curvature is defined by 
\begin{equation}
K_F(z,v)=\frac{2\langle\Omega(\chi,\bar\chi)\chi,\chi\rangle_g}{G^2}
=\frac2{G^2}G_iR^i.\label{eq:2.17}
\end{equation}
Thus the identity
\begin{equation}
R^i=cGv^i\quad\text{on }\Mt\label{eq:2.18}
\end{equation}
implies \(K_F=2c\) by \eqref{eq:2.17}.  Therefore, \eqref{eq:2.18} is the invariant identity
\(\Omega(\chi,\bar\chi)\chi=cG\chi\).

\begin{definition}
A strongly pseudoconvex complex Finsler metric \(F\) is said to have  
\emph{constant radial curvature} \(c\), for some \(c\in\R\), if
\eqref{eq:2.18} holds.
\end{definition}

\begin{lemma}\label{lem:r}
For a strongly pseudoconvex complex Finsler metric,
we have
\begin{equation}
r_{j\bar k}=-G_{;j\bar k}+g_{a\bar b}N^a_j\overline{N^b_k},
\qquad r_{j\bar k}=\overline{r_{k\bar j}},
\label{eq:2.19}
\end{equation}
and
\begin{equation}
R^i=-\bar\chi(S^i).\label{eq:2.20}
\end{equation}
\end{lemma}
\begin{proof}
Since   \(G_iN^i_j=G_{;j}\) ,  \(\delta_{\bar k}G_i=0\)  and  \(R^i_{j\bar k}=-\delta_{\bar k}N^i_j\), we get 
\begin{equation}
-r_{j\bar k}=G_i\delta_{\bar k}N^i_j=\delta_{\bar k}G_{;j}
=G_{;j\bar k}-\overline{N^a_k}G_{\bar a;j}.
\label{eq:2.21}
\end{equation}
It follows from \eqref{eq:2.4} that
\(G_{\bar a;j}=g_{i\bar a}N^i_j\). Substituting this into
\eqref{eq:2.21}, we have
\[
r_{j\bar k}=-G_{;j\bar k}+g_{a\bar b}N^a_j\overline{N^b_k}.
\]
Since  \(\overline{G_{;k\bar j}}=G_{;j\bar k}\) and 
 \(\overline{g_{a\bar b}}=\vp_b\vp_{\bar a}G=g_{b\bar a}\), we get 
\[
\overline{r_{k\bar j}}=-\overline{G_{;k\bar j}}
+\overline{g_{a\bar b}}N^b_j\overline{N^a_k}
=-G_{;j\bar k}+g_{b\bar a}N^b_j\overline{N^a_k}=r_{j\bar k}.
\]
This proves \eqref{eq:2.19}.

Since  
\begin{equation}
\delta_{\bar k}v^j=\partial_{\bar k}v^j-\overline{N^r_k}\vp_{\bar r}v^j=0,
\label{eq:2.22}
\end{equation}
we obtain 
\[
\bar\chi(S^i)=\bar v^k\delta_{\bar k}(N^i_jv^j)
=\bar v^kv^j\delta_{\bar k}N^i_j=-R^i_{j\bar k}v^j\bar v^k=-R^i.
\]
 This proves \eqref{eq:2.20}.
\end{proof}

\begin{proposition}[ {\cite{AP}}]\label{prop:AP}
For a strongly pseudoconvex K\"ahler--Finsler metric and a real constant \(c\),
the conditions \(K_F=2c\) and \eqref{eq:1.2} together are equivalent to
\eqref{eq:2.18}.
\end{proposition}

\begin{lemma}[\cite{Ai95,XiaoQiuZhong}]\label{lem:normal}
If \(F\) is a K\"ahler--Finsler metric , then for each \((p,v_0)\in\Mt\) there are
holomorphic coordinates such that
\begin{equation}
g_{i\bar j}(p,v_0)=\delta_{ij},\qquad\Gamma^i_{jk}(p,v_0)=0,\qquad N^i_j(p,v_0)=0.
\label{eq:2.23}
\end{equation}
\end{lemma}
\begin{proof}
We choose coordinates \(z\) centered at \(p\) with \(g(p,v_0)=I\), and 
make a  holomorphic coordinate change
\[
z^i=w^i+\frac12q^i_{jk}w^jw^k,\qquad q^i_{jk}=q^i_{kj}=-\Gamma^i_{jk}(p,v_0).
\]
Write \(L^i_j=\partial z^i/\partial w^j\), so the fiber coordinates satisfy
\(v=Lu\). The connection transformation law is
\[
\widetilde\Gamma^i_{jk}=(L^{-1})^i_aL^b_jL^d_k\Gamma^a_{bd}
+(L^{-1})^i_a\frac{\partial^2z^a}{\partial w^j\partial w^k}.
\]
Since   \( (L^i_j)=I\) at the chosen point $(p,v_0)$,  we have
$$\widetilde\Gamma^i_{jk}(p,v_0)=\Gamma^i_{jk}(p,v_0)+q^i_{jk}=0 \mbox{ and  }   \widetilde N^i_j(p,v_0)=v_0^k \widetilde\Gamma^i_{jk}(p,v_0)=0.$$
\end{proof}

\section{Local rigidity and the classification theorem}
Let \(V\subset M\) be an open subset of a complex manifold $M$ and let \(C\subset T^{1,0}V\setminus\{0\}\) be open and invariant
under fiber multiplication by \(\C^*\). Assume that every fiber \(C_z\) is
nonempty and connected. 

\begin{theorem}\label{theo:cone}
Suppose that \(F\) is a strongly pseudoconvex K\"ahler--Finsler metric on \(C\), and
\(R^i=cGv^i\)   for a real constant \(c\ne0\). Then there exists a unique smooth
Hermitian K\"ahler metric \(h\) on \(V\) satisfying
\begin{equation}
G(z,v)=h_z(v,v)\qquad((z,v)\in C).\label{eq:3.1}
\end{equation}
Its holomorphic sectional curvature is \(2c\) in every nonzero direction.
\end{theorem}

\subsection{Some identities of K\"ahler--Finsler metrics}
The next identities \eqref{eq:3.3} and  \eqref{eq:3.4} are needed for our
argument.

\begin{lemma}\label{lem:AB}
For every strongly pseudoconvex complex Finsler metric on \(C\), we have
\begin{align}
G_iB^i_{\bar\ell}&=0,\label{eq:3.2}\\
g_{i\bar k}B^i_{\bar\ell}&=g_{i\bar\ell}B^i_{\bar k}.
\label{eq:3.3}
\end{align}
If the metric is K\"ahler--Finsler, then
\begin{equation}
A_{ki}B^i_{\bar\ell}=0.\label{eq:3.4}
\end{equation}
Thus, in coordinates with \(g=I\), the matrices \(A,B\) are symmetric and
\(AB=0\).
\end{lemma}
\begin{proof}
We differentiate \(G_iN^i_j=G_{;j}\) in \(\bar v^\ell\). Since
\(g_{i\bar\ell}N^i_j=G_{\bar\ell;j}\), it follows from the subtraction that
\[
G_i\vp_{\bar\ell}N^i_j=0.
\]
Contracting with \(v^j\), we obtain \eqref{eq:3.2}.
Differentiating \eqref{eq:3.2} with respect to \(\bar v^k\),
we get
\[
g_{i\bar k}B^i_{\bar\ell}
+G_i\vp_{\bar k}\vp_{\bar\ell}S^i=0.
\]
We exchange \(k\) and \(\ell\) and subtract to get \eqref{eq:3.3}.

Suppose now that the metric is K\"ahler--Finsler. By Lemma~\ref{lem:strong},
\[
\delta_jg_{k\bar\ell}=\delta_kg_{j\bar\ell},\qquad
\delta_jG_k=\delta_kG_j.
\]
Differentiating the second equality in \(\bar v^\ell\) and using the first
equality, we get
\[
A_{kr}\vp_{\bar\ell}N^r_j=A_{jr}\vp_{\bar\ell}N^r_k.
\]
After contraction with \(v^j\), it follows from \(A_{jr}v^j=0\) that  the right side vanishes,
while the left side becomes \eqref{eq:3.4}.
\end{proof}

\begin{lemma}\label{lem:M}
Suppose that \(F\) is  a K\"ahler--Finsler metric on \(C\) and
\[
R^i=cGv^i
\]
for a real constant \(c\) throughout \(C\). Then
\begin{equation}
r_{j\bar k}v^j=cGG_{\bar k}.\label{eq:3.5}
\end{equation}
At a point where \eqref{eq:2.23} holds, set
\(M^i_{\bar k}=\delta_{\bar k}S^i\). Then
\begin{equation}
M^i_{\bar k}=-c\bigl(v^iG_{\bar k}+G\overline{A_{ik}}\bigr).
\label{eq:3.6}
\end{equation}
\end{lemma}
\begin{proof}
Fix \((p,v_0)\in C\) and choose normal coordinates as in
Lemma~\ref{lem:normal}.
Unless otherwise stated, the following identities are evaluated
at \((p,v_0)\).
Set \(P_{j\bar k}=G_{;j\bar k}\).
Since \(N^i_j=0\), from the curvature condition and \eqref{eq:2.20}, we have
\begin{equation}
M^i_{\bar k}\bar v^k=-cGv^i,\qquad
M^i_{\bar k}=v^j\partial_{\bar k}N^i_j.
\label{eq:3.7}
\end{equation}
Differentiate the strong K\"ahler identity
\(\delta_jg_{\ell\bar s}=\delta_\ell g_{j\bar s}\) in \(\bar z^k\).
Because \(N^i_j=0\) at the point, we obtain
\[
P_{j\bar k,\ell\bar s}-P_{\ell\bar k,j\bar s}
=(\partial_{\bar k}N^r_j)G_{\ell r\bar s}
-(\partial_{\bar k}N^r_\ell)G_{jr\bar s}.
\]
Here the comma indicates fiber differentiation of \(P\). Contracting first
with \(\bar v^s\) and then with \(v^j\), we get
\begin{equation}
v^j\vp_\ell P_{j\bar k}=P_{\ell\bar k}+A_{\ell r}M^r_{\bar k}.
\label{eq:3.8}
\end{equation}
Here one first uses
\[
\bar v^sG_{\ell r\bar s}=A_{\ell r},\qquad
\bar v^sP_{j\bar k,\ell\bar s}=\vp_\ell P_{j\bar k},
\]
and then \(v^j\vp_jP_{\ell\bar k}=P_{\ell\bar k}\) and \(v^jA_{jr}=0\).

Since \(P\) is Hermitian, conjugating \eqref{eq:3.8}, we have
\[
\bar v^s\vp_{\bar\ell}P_{j\bar s}
=P_{j\bar\ell}+\overline{A_{\ell r}M^r_{\bar j}}.
\]
Contracting with $v^j$ and using the conjugate of \eqref{eq:3.7}, we find
\[
v^j\overline{A_{\ell r}M^r_{\bar j}}
=\overline{A_{\ell r}}\,
  \overline{\bar v^jM^r_{\bar j}}
=-cG\overline{A_{\ell r}}\bar v^r=0.
\]
Therefore,
\begin{equation}
v^j\bar v^s\vp_{\bar\ell}P_{j\bar s}=v^jP_{j\bar\ell}.
\label{eq:3.9}
\end{equation}
It follows from \eqref{eq:2.15}  and   \eqref{eq:2.19} that 
\begin{equation}
r_{j\bar k}v^j\bar v^k=cG^2=-P_{j\bar k}v^j\bar v^k+g_{a\bar b}S^a\overline{S^b} \label{eq:3.10}
\end{equation}
on \(C\). 
At point \((p,v_0)\), \(S^a=0\)  and $\vp_{\bar\ell} \overline{S^b}=0$. Differentiating \eqref{eq:3.10} in \(\bar v^\ell\) and using
\eqref{eq:3.9}, we obtain
\[
-2v^jP_{j\bar\ell}=2cGG_{\bar\ell}.
\]
Since \(r_{j\bar k}=-P_{j\bar k}\) at point \((p,v_0)\), this is \eqref{eq:3.5}.
Both sides of \eqref{eq:3.5} are covectors. Since $(p,v_0)$ is
arbitrary, \eqref{eq:3.5} holds on all $C$. 

Differentiating \eqref{eq:3.5} with respect to \(\bar v^i\), we obtain
\[
v^j\vp_{\bar i}r_{j\bar k}
=c\bigl(G_{\bar i}G_{\bar k}+GG_{\bar i\bar k}\bigr).
\]
At the chosen point, \(N^i_j=0\), so every term in
\[
\vp_{\bar i}
\bigl(g_{a\bar b}N^a_j\overline{N^b_k}\bigr)
\]
vanishes. Hence \eqref{eq:2.19} gives
\[
\vp_{\bar i}r_{j\bar k}
=-\vp_{\bar i}P_{j\bar k}.
\]
 Differentiating
the  identity \(G_{\bar i;j}=g_{a\bar i}N^a_j\) in the
{base} variable \(\bar z^k\), we get 
\[
\vp_{\bar i}P_{j\bar k}
=(\partial_{\bar k}g_{a\bar i})N^a_j
+g_{a\bar i}\partial_{\bar k}N^a_j.
\]
We use \(N^a_j=0\) and \(g=I\) at point $(p,v_0)$, and obtain
\[
v^j\vp_{\bar i}P_{j\bar k}
=v^j\partial_{\bar k}N^i_j=M^i_{\bar k}.
\]
It follows that
\[
-M^i_{\bar k}=c\bigl(G_{\bar i}G_{\bar k}+GG_{\bar i\bar k}\bigr).
\]
At point $(p,v_0)$, \(G_{\bar i}=v^i\) and
\(G_{\bar i\bar k}=\overline{A_{ik}}\), that is, \eqref{eq:3.6} is valid.
\end{proof}
\begin{remark}
Suppose that $F$ is induced by a Hermitian K\"ahler metric,
$G=h_{a\bar b}(z)v^a\bar v^b$. Then $A=0$, and
$S^i=\Gamma^i_{ab}(z)v^av^b$ is holomorphic in $v$, so $B=0$.
The radial identity \eqref{eq:2.18} becomes
\begin{equation}
-\partial_{\bar k}\Gamma^i_{ab}\,v^av^b\bar v^k
=c\,h_{a\bar b}v^av^i\bar v^b.
\label{eq:3.11}
\end{equation}
This is an identity in all fiber variables. Differentiating first in
$\bar v^k$ and then in $v^j$ and $v^\ell$, and using
$\Gamma^i_{j\ell}=\Gamma^i_{\ell j}$, we get
\[
-2\partial_{\bar k}\Gamma^i_{j\ell}
=c\bigl(\delta^i_jh_{\ell\bar k}
       +\delta^i_\ell h_{j\bar k}\bigr).
\]
Consequently,
\begin{equation}
R_{j\bar a\ell\bar k}
=\frac c2\bigl(h_{j\bar a}h_{\ell\bar k}
             +h_{\ell\bar a}h_{j\bar k}\bigr).
\label{eq:3.12}
\end{equation}
Conversely, contracting \eqref{eq:3.12} gives \eqref{eq:2.18}. Thus, in the Hermitian K\"ahler case, the radial
identity \eqref{eq:2.18} is equivalent to constant holomorphic sectional curvature
$2c$.
\end{remark}

\subsection{Proof of Theorem 3.1 and the global conclusions}
The identity \eqref{eq:3.13} is the coordinate identity that underlies the calculations of curvature and torsion in \cite{AP97}. We derive it directly.

\begin{lemma}\label{lem:rigidity}
Under the assumptions of Lemma~\ref{lem:M}, at the point in the normal
coordinates \eqref{eq:2.23},
\begin{equation}
\overline B B+cGA\overline A=0.\label{eq:3.13}
\end{equation}
Equivalently, because \(B^T=B\),
\begin{equation}
B^*B=-cGA\overline A.\label{eq:3.14}
\end{equation}
\end{lemma}
\begin{proof}
Differentiating the identity \(\delta_{\bar k}G_i=0\) in \(v^j\),
we get, in any coordinates,
\[
\delta_{\bar k}A_{ij}=g_{i\bar a}\vp_j\overline{N^a_k}.
\]
After contraction with \(\bar v^k\),
\begin{equation}
\bar\chi A_{ij}=g_{i\bar a}\overline{B^a_{\bar j}}.
\label{eq:3.15}
\end{equation}
Thus \(\bar\chi A=\overline B\) at the normal point, where \(g=I\).

Now we differentiate the cone identity \(\bar\chi S^i=-cGv^i\) in
\(\bar v^\ell\). The right side gives
\begin{equation}
\vp_{\bar\ell}(\bar\chi S^i)
=-c\vp_{\bar\ell}(Gv^i)=-cG_{\bar\ell}v^i,
\label{eq:3.16}
\end{equation}
because \(\vp_{\bar\ell}v^i=0\). For the left side, write \(\bar\chi\) in
coordinates,
\[
\bar\chi=\bar v^k\delta_{\bar k}
=\bar v^k\partial_{\bar k}-\overline{S^a}\vp_{\bar a},
\qquad \overline{S^a}=\bar v^k\overline{N^a_k},
\]
and compute its commutator with \(\vp_{\bar\ell}\). Since
\(\vp_{\bar\ell}\bar v^k=\delta^k_\ell\) and
\([\vp_{\bar\ell},\vp_{\bar a}]=0\),
\[
[\vp_{\bar\ell},\bar\chi]
=\partial_{\bar\ell}-(\vp_{\bar\ell}\overline{S^a})\vp_{\bar a}.
\]
Conjugating the K\"ahler identity \(\vp_kS^i=2N^i_k\) of \eqref{eq:2.12}, we obtain
 \(\vp_{\bar\ell}\overline{S^a}=2\overline{N^a_\ell}\). Hence
\[
[\vp_{\bar\ell},\bar\chi]
=\partial_{\bar\ell}-2\overline{N^a_\ell}\vp_{\bar a}
=\delta_{\bar\ell}-\overline{N^a_\ell}\vp_{\bar a}.
\]
Applying this operator identity to \(S^i\) and using
\(\vp_{\bar a}S^i=B^i_{\bar a}\), we have
\[
\begin{aligned}
\vp_{\bar\ell}(\bar\chi S^i)
&=(\delta_{\bar\ell}-\overline{N^a_\ell}\vp_{\bar a})S^i
+\bar\chi(\vp_{\bar\ell}S^i)\\
&=\delta_{\bar\ell}S^i-\overline{N^a_\ell}B^i_{\bar a}
+\bar\chi B^i_{\bar\ell}.
\end{aligned}
\]
At the normal point \(N=0\), the middle term vanishes, and
\(\delta_{\bar\ell}S^i=M^i_{\bar\ell}\). Comparing with \eqref{eq:3.16}, we get
\begin{equation}
M^i_{\bar\ell}+\bar\chi B^i_{\bar\ell}=-cG_{\bar\ell}v^i.
\label{eq:3.17}
\end{equation}
Combining \eqref{eq:3.17} with \eqref{eq:3.6}, at the normal point, we have
\begin{equation}
\bar\chi B^i_{\bar\ell}=cG\overline{A_{i\ell}}.
\label{eq:3.18}
\end{equation}
The identity $A_{ki}B^i_{\bar\ell}=0$ holds in the coordinate
neighborhood. Applying $\bar\chi$ before evaluating at the normal
point, we get
\[
0=(\bar\chi A_{ki})B^i_{\bar\ell}
  +A_{ki}\bar\chi B^i_{\bar\ell}.
\]
At the normal point, it follows from \eqref{eq:3.15} and \eqref{eq:3.18} that
\[
0=\overline{B^k_{\bar i}}B^i_{\bar\ell}
  +cGA_{ki}\overline{A_{i\ell}}.
\]
Hence $\overline B B+cGA\overline A=0$.
Since $B^T=B$ at the normal point, $B^*=\overline B$. Therefore,
$B^*B=-cGA\overline A$.
\end{proof}

\begin{lemma}\label{lem:algebra}
Let \(A,B\) be complex symmetric square matrices and let
\(\lambda\in\R\setminus\{0\}\). If
\[
AB=0,\qquad B^*B=\lambda A\overline A,
\]
then \(A=B=0\).
\end{lemma}
\begin{proof}
For symmetric matrices, \(A^*=\overline A\) and \(B^*=\overline B\).
Conjugating the second identity, we get \(BB^*=\lambda A^*A\).
Therefore
\[
0=\tr\bigl((AB)(AB)^*\bigr)
=\tr(ABB^*A^*)=\lambda\tr\bigl((AA^*)^2\bigr).
\]
The matrix \(AA^*\) is Hermitian positive semidefinite. Since
\(\lambda\ne0\), all its eigenvalues vanish, and \(A=0\).
The equation for \(B^*B\) then yields \(B=0\).
\end{proof}

\begin{proof}[Proof of Theorem~\ref{theo:cone}]
Fix a point of \(C\) and choose normal  coordinates as in
Lemma~\ref{lem:normal}. Lemmas~\ref{lem:AB} and
\ref{lem:rigidity} give the assumptions of
Lemma~\ref{lem:algebra} with \(\lambda=-cG\ne0\).
Thus \(G_{ij}=0\) at this point. This condition is independent
of the coordinates. Since the point was arbitrary,
\[
G_{ij}=0\qquad\text{on }C.
\]
Differentiating this identity and its complex conjugate, we get
\[
\vp_k g_{i\bar j}
=\vp_{\bar j}G_{ik}=0,
\qquad
\vp_{\bar k}g_{i\bar j}
=\vp_iG_{\bar j\bar k}=0.
\]
Since \(C_z\) is connected, \(g_{i\bar j}(z,v)\) is independent
of \(v\in C_z\). We therefore set
\[
h_{i\bar j}(z)=g_{i\bar j}(z,v),
\qquad v\in C_z.
\]

We next show that \(h\) is smooth. Fix \(z_0\) and choose local
coordinates and a vector \(v_0\in C_{z_0}\). Since \(C\) is open,
the same coordinate vector \(v_0\) belongs to \(C_z\) for all
\(z\) near \(z_0\). Hence
\[
h_{i\bar j}(z)=g_{i\bar j}(z,v_0)
\]
is smooth. The matrix \(h\) is positive definite and Hermitian
because \(g\) is. Its components also have the same transformation
rule as those of \(g\). Thus \(h\) defines a Hermitian metric on \(V
\).
Euler's identity gives
\[
G(z,v)=g_{i\bar j}(z,v)v^i\bar v^j
=h_{i\bar j}(z)v^i\bar v^j
\qquad ((z,v)\in C),
\]
which proves \eqref{eq:3.1}.
Any other Hermitian metric satisfying \eqref{eq:3.1} must have
components \(G_{i\bar j}=h_{i\bar j}\) on \(C\), so \(h\) is unique.

On \(C\), the K\"ahler condition becomes
\[
\partial_jh_{k\bar\ell}
=\partial_kh_{j\bar\ell}.
\]
Thus \(h\) is K\"ahler. Since \(g\) is independent of \(v\),
the Chern--Finsler connection of \(F\) agrees on \(C\) with the
pullback of the Chern connection of \(h\). The radial curvature
identity then gives
\[
K_h(z,v)=2c\qquad(v\in C_z).
\]

For fixed \(z\), the expression
\[
\bigl(K_h(z,v)-2c\bigr)h_z(v,v)^2
\]
extends to a polynomial in \(v,\bar v\) of bidegree \((2,2)\).
It vanishes on the nonempty open set \(C_z\), and hence vanishes
everywhere. Therefore \(K_h(z,v)=2c\) for every nonzero \(v\).
\end{proof}

\begin{proof}[Proof of Theorem~\ref{thm:main}]
Take \(V=M\) and \(C=T^{1,0}M\setminus\{0\}\). Each fiber
\(C_p=T^{1,0}_pM\setminus\{0\}\) is nonempty and connected.
The radial curvature assumption gives \(R^i=cGv^i\) on \(C\).
Thus, by Theorem~\ref{theo:cone}, we obtain a smooth Hermitian
K\"ahler metric \(h\) on \(M\) such that
\[
F(p,v)^2=h_p(v,v)
\qquad
(p\in M,\ v\in T^{1,0}_pM\setminus\{0\}).
\]
Moreover, \(h\) has constant holomorphic sectional curvature \(2c\).
This proves the theorem.
\end{proof}

\begin{proof}[Proof of Theorem~\ref{thm:classification}]
By Proposition~\ref{prop:AP}, the assumptions \(K_F=2c\) and
\eqref{eq:1.2} imply the curvature identity \eqref{eq:1.1}.
It follows from Theorem~\ref{thm:main} that \(F\) is induced by a
Hermitian K\"ahler metric \(h\) of holomorphic sectional curvature \(2c\).

Now assume that \(c<0\), \(M\) is simply connected, and \(F\) is complete.
Since \(F(z;v)^2=h_z(v,v)\), the metric \(h\) is also complete.
Using the classification of complete simply connected K\"ahler manifolds
with constant holomorphic sectional curvature (\cite{KN}), we obtain a
biholomorphic isometry
\[
(M,F)\cong(\B^n,a\kappa_{\B^n})
\]
for some \(a>0\). Since \(\kappa_{\B^n}\) has holomorphic sectional
curvature \(-4\), the  metric \(a\kappa_{\B^n}\) has curvature
\(-4/a^2\). Hence
\[
-\frac{4}{a^2}=2c,
\qquad
a=\sqrt{-\frac{2}{c}},
\]
which proves the assertion.
\end{proof}

\begin{remark}
For \(c=0\), from \eqref{eq:3.14} it follows that \(B=0\), but it does not force \(A=0\).
Indeed, for \(n\ge2\) and sufficiently small \(\varepsilon>0\),
\[
G(v)=|v|^2+\varepsilon\frac{|v^1|^4}{|v|^2}
\]
is the square of a smooth positive, nonquadratic complex Minkowski norm;
its fiber Levi tensor is positive definite by the compactness of the unit
sphere. It is independent of the base point, so \(N=0\), its horizontal
torsion vanishes, and its radial curvature is zero.
\end{remark}

\section{Proof of Theorem  1.3}
Let \(\D=\{\zeta\in\C:|\zeta|<1\}\) be the unit disk.
We use the normalization
\[
\kappa_{\D}(\zeta;\xi)=\frac{|\xi|}{1-|\zeta|^2}.
\] 
\begin{lemma}\label{lem:retract}
Let \(D\subset\C^n\) be a domain. Suppose that \(\kappa_D\) is smooth and
strongly pseudoconvex near \((p,v)\in T^{1,0}D\setminus\{0\}\). If there
are holomorphic maps
\[
\varphi:\D\to D,\quad q:D\to\D,\quad
q\circ\varphi=\id_{\D},\quad\varphi(0)=p,\quad
v\in\C^*\varphi'(0),
\]
then, at \((p,v)\), we  have
\[
v^jG_{k;j}-G_{;k}-S^rA_{kr}=0\qquad(k=1,\dots,n).
\]
and
\begin{equation}
R^i=-2Gv^i.\label{eq:4.1}
\end{equation}
\end{lemma}
\begin{proof}
By contractivity under \(\varphi\) and \(q\), and
\(q\circ\varphi=\id_{\D}\), we have
\[
\kappa_D\bigl(\varphi(\zeta);\xi\varphi'(\zeta)\bigr)
=\frac{|\xi|}{1-|\zeta|^2}.
\]
Since $q\circ\varphi=\mathrm{id}_{\mathbb D}$, we get $dq_p\ne0$.
Choose local holomorphic coordinates near $p$ with $z^1=q$.
After shrinking the coordinate neighborhood, subtract the transverse
components of $\varphi(z^1)$ from the remaining coordinates. Then
$\varphi(\zeta)=(\zeta,0,\ldots,0)$ for $\zeta$ near $0$.
By homogeneity, it suffices to prove the assertions for $v=e_1$.

Set \(G=\kappa_D^2\) and
\[
P(z,w)=\frac{|w^1|^2}{(1-|z^1|^2)^2}.
\]
It follows from the contractivity under \(q\) that \(G\ge P\), with equality at
\(((\zeta,0),e_1)\). Thus \(G-P\) has a local minimum there, so its first
partial derivatives in all real coordinate directions vanish. Writing
\(a(\zeta)=(1-|\zeta|^2)^{-2}\), we obtain along the lifted disk
\[
G=a,\quad G_i=G_{\bar i}=\delta_{i1}a,\quad
G_{;k}=\delta_{k1}\partial_\zeta a.
\]
From the Euler identity \(G_i=g_{i\bar j}\bar v^j\), we have
\(g_{i\bar1}=g_{1\bar i}=\delta_{i1}a\) at \(v=e_1\).

Using \(N^i_1=g^{\bar s i}\partial_1G_{\bar s}\) and differentiating
\(G_{\bar s}=\delta_{s1}a\) along the disk, we get
\[
N^i_1=\delta_{i1}\frac{\partial_\zeta a}{a}
=\delta_{i1}\frac{2\bar\zeta}{1-|\zeta|^2}.
\]
At $((0,\ldots,0),e_1)$, we have $N^i_1=0$ for every $i$.
Since $S^i=N^i_jv^j=N^i_1$ at $v=e_1$, we also have $S^i=0$.
Moreover, $\partial_1G_k=0$ and $G_{;k}=0$ there.
 Hence the weak K\"ahler condition
\[
v^jG_{k;j}-G_{;k}-S^rA_{kr}=0
\]
holds there.

Finally, since \(N^i_1=0\) at \((0,e_1)\),
\[
\delta_{\bar1}N^i_1=\partial_{\bar1}N^i_1=2\delta_{i1}.
\]
By \eqref{eq:2.15}, \(R^i=-2\delta_{i1}=-2Gv^i\) at \((0,e_1)\).
Both conclusions at \((p,v)\) follow by coordinate invariance
and homogeneity.
\end{proof}

\begin{proof}[Proof of Theorem~\ref{thm:convex}]
By Lempert's theory \cite{Lempert81,Lempert82}, the Kobayashi metric
\(\kappa_D\) is smooth and strongly pseudoconvex, and there is a holomorphic
retract through every point and nonzero direction. By Lemma~\ref{lem:retract},
\(\kappa_D\) is weakly K\"ahler--Finsler and satisfies \eqref{eq:4.1}
everywhere.

We now begin to prove (i) \(\Rightarrow\) (ii).   Assume (i). Applying Theorem~\ref{thm:main} to \(F=\kappa_D\) with \(c=-2\), we obtain
 that \(\kappa_D\) is Hermitian, proving (ii).

(ii) \(\Rightarrow\) (iii).  We assume (ii), and write $G=h_{a\bar b}(z)v^a\bar v^b$.
Then $A=0$. From the weak K\"ahler condition
\eqref{eq:2.13}, it follows that
\[
\bigl(\partial_jh_{k\bar\ell}
     -\partial_kh_{j\bar\ell}\bigr)v^j\bar v^\ell=0
\qquad\text{for every }v.
\]
Differentiating this polynomial identity in $v^j$ and $\bar v^\ell$, we get
$\partial_jh_{k\bar\ell}=\partial_kh_{j\bar\ell}$.
Thus $h$ is K\"ahler. The radial identity \eqref{eq:4.1} gives
holomorphic sectional curvature $-4$. Moreover,
\(\kappa_D\) is complete, and \(D\) is simply connected because it is
convex. By means of the classical classification of complete simply connected
K\"ahler manifolds of constant holomorphic sectional curvature, we obtain
that \((D,\kappa_D)\) is biholomorphically isometric to the unit ball,
proving (iii).

(iii) \(\Rightarrow\) (i).  Assume (iii), and let \(f:D\to\B^n\) be biholomorphic. The
Kobayashi metric is invariant under biholomorphisms, and on the ball it has
the explicit formula
\[
\kappa_{\B^n}(z;v)
=\frac{\sqrt{(1-\|z\|^2)\|v\|^2+|\langle v,z\rangle|^2}}{1-\|z\|^2}.
\]
This is a Hermitian K\"ahler metric. Hence \(\kappa_D=f^*\kappa_{\B^n}\)
is K\"ahler--Finsler, proving (i).
\end{proof}

\begin{corollary}\label{cor:scale}
Let \(F\) be a smooth strongly pseudoconvex complex Finsler metric on
a domain \(D\). Suppose that every nonzero tangent direction admits
a retract as in Lemma~\ref{lem:retract}, and that
\[
\varphi^*F=a_0\kappa_{\D}
\]
for every selected disk, with the same constant \(a_0>0\).
Then \(F=a_0\kappa_D\), \(F\) is weakly K\"ahler--Finsler, and
\[
R^i=-\frac{2}{a_0^2}Gv^i,\qquad G=F^2.
\]
If \(F\) is K\"ahler--Finsler, then it is Hermitian K\"ahler with
holomorphic sectional curvature \(-4/a_0^2\).
\end{corollary}

\begin{proof}
Fix \(p\in D\) and \(v\ne0\). Choose a retract
\(\varphi:\D\to D\), \(q:D\to\D\) such that
\(\varphi(0)=p\) and \(v=\lambda\varphi'(0)\).
Since \(q\circ\varphi=\operatorname{id}_{\D}\), it follows from the contractivity that
\[
|\lambda|
=\kappa_{\D}(0;dq_p(v))
\le\kappa_D(p;v)
\le\kappa_{\D}(0;\lambda)
=|\lambda|.
\]
Thus \(\kappa_D(p;v)=|\lambda|\). From the assumption
\(\varphi^*F=a_0\kappa_{\D}\), we have
\(F(p;v)=a_0|\lambda|\), so \(F=a_0\kappa_D\).

In particular, \(\kappa_D=F/a_0\) is smooth and strongly
pseudoconvex. By applying Lemma~\ref{lem:retract} to
\(\kappa_D\), it follows that \(\kappa_D\) is weakly K\"ahler--Finsler and satisfies
\[
R_{\kappa_D}^i=-2\kappa_D^2v^i.
\]
Multiplication by a positive constant leaves the Chern--Finsler
connection and the weakly K\"ahler--Finsler condition unchanged.
Hence \(F\) is weakly K\"ahler--Finsler and
\[
R_F^i=R_{\kappa_D}^i
=-2\kappa_D^2v^i
=-\frac{2}{a_0^2}Gv^i.
\]

Now assume that \(F\) is K\"ahler--Finsler. Since
\(\chi=v^i\delta_i\), from the preceding identity, it follows that
\[
\Omega(\chi,\bar\chi)\chi
=R_F^i\delta_i
=-\frac{2}{a_0^2}G\chi.
\]
Thus \(F\) satisfies the hypothesis of Theorem~\ref{thm:main}
with \(c=-2/a_0^2\). It follows that \(F\) is Hermitian
K\"ahler with holomorphic sectional curvature \(2c=-4/a_0^2\).
\end{proof}

\section{Proof of Theorem  1.4}
We now prove Theorem~\ref{thm:boundary}. The main point is to recover a
Hermitian K\"ahler metric from the Kobayashi metric near the whole boundary.

Put \(\delta(z)=\operatorname{dist}(z,\partial D)\).
For \(z\) sufficiently close to \(\partial D\), let \(\pi(z)\) be the nearest
boundary point and write
\[
v=v_T+v_N,\qquad v_T\in T^\C_{\pi(z)}\partial D,\quad
v_N\perp T^\C_{\pi(z)}\partial D.
\]
By Bracci--Forn{\ae}ss--Wold  \cite{BFW}, there exist
\(\epsilon,\eta>0\) such that every nonzero vector satisfying
\[
0<\delta(z)<\epsilon,\qquad |v_N|<\eta|v_T|
\]
is tangent to a holomorphic disk, which is a holomorphic retract of \(D\).
Set
\begin{equation}
C=\{(z,v):0<\delta(z)<\epsilon,\ |v_N|<\eta|v_T|\}.
\label{eq:5.1}
\end{equation}
Then every fiber \(C_z\) is a nonempty connected open cone.

Assume condition \textup{(ii)} of Theorem~\ref{thm:boundary}.
By decreasing \(\epsilon\), we may assume that \(\kappa_D\) is a
smooth strongly pseudoconvex K\"ahler--Finsler on \(V_\epsilon\).
For every \((z,v)\in C\), from Lemma~\ref{lem:retract}, it follows that
\(R^i=-2Gv^i\).
Applying Theorem~\ref{theo:cone} with \(G=\kappa_D^2\) and \(c=-2\),
we obtain a smooth Hermitian K\"ahler metric \(h\) on \(V_\epsilon\)
such that
\begin{equation}
G(z,v)=h_z(v,v)\qquad((z,v)\in C)\label{eq:5.2}
\end{equation}
and
\begin{equation}
K_h\equiv-4.\label{eq:5.3}
\end{equation}
The equality \eqref{eq:5.2} is needed only on the cone $C$; the curvature
identity \eqref{eq:5.3} holds in every nonzero direction.

A metric defined outside a compact set is called
\emph{complete at infinity} if it agrees, outside a possibly
larger compact set, with a complete smooth positive definite
metric on the whole manifold.

\begin{lemma}\label{lem:infinity}
Let \(h\) be a smooth positive definite Hermitian metric on a boundary collar
of \(D\) satisfying \eqref{eq:5.2} for some fixed \(\eta>0\). Then \(h\) is
complete at infinity.
\end{lemma}
\begin{proof}
By the boundary estimates of Balogh--Bonk \cite{BB}, after
shrinking the collar $V_{\epsilon}$, there are constants \(b,B>0\) such that
\begin{equation}
\kappa_D(z;w)\ge b\frac{|w_N|}{\delta(z)},\qquad
\kappa_D(z;t)\le B\frac{|t|}{\sqrt{\delta(z)}},
\label{eq:5.4}
\end{equation}
where \(t\) is complex tangential. Write \(\|v\|_{h,z}=\sqrt{h_z(v,v)}\).

Set \(a=\eta/2\). At a fixed \(z\), let \(\nu\) and \(e\) be Euclidean unit
complex-normal and complex-tangential vectors, respectively. Since both
\(e\) and \(e+a\nu\) belong to \(C_z\), from \eqref{eq:5.2} and \eqref{eq:5.4},
it follows that
\[
\begin{aligned}
a\|\nu\|_{h,z}
&\ge\|e+a\nu\|_{h,z}-\|e\|_{h,z}\\
&=\kappa_D(z;e+a\nu)-\kappa_D(z;e)\\
&\ge\frac{ba}{\delta(z)}-\frac B{\sqrt{\delta(z)}}.
\end{aligned}
\]
After shrinking the collar  $V_{\epsilon}$, we have
\begin{equation}
\|\nu\|_{h,z}\ge\frac b{2\delta(z)}.\label{eq:5.5}
\end{equation}
If \(v\notin C_z\), then \(|v_T|\le|v_N|/\eta\).
For a nonzero complex-tangential vector \(v_T\), it follows from \eqref{eq:5.2} that
\(\|v_T\|_{h,z}=\kappa_D(z;v_T)\), and the same equality holds trivially
when \(v_T=0\). Using \eqref{eq:5.5}, \eqref{eq:5.4}, and the triangle
inequality, we obtain
\[
\begin{aligned}
\|v\|_{h,z}
&\ge\|v_N\|_{h,z}-\|v_T\|_{h,z}\\
&\ge\frac{b|v_N|}{2\delta(z)}-\frac{B|v_T|}{\sqrt{\delta(z)}}\\
&\ge\left(\frac b2-\frac B\eta\sqrt{\delta(z)}\right)
\frac{|v_N|}{\delta(z)}.
\end{aligned}
\]
If \(v\in C_z\), then it follows from \eqref{eq:5.2} and \eqref{eq:5.4} that
\[
\|v\|_{h,z}=\kappa_D(z;v)\ge b\frac{|v_N|}{\delta(z)}.
\]
Thus, after another uniform shrinking of the collar $V_{\epsilon}$, there is a constant
\(b_1>0\) such that
\begin{equation}
\|v\|_{h,z}\ge b_1\frac{|v_N|}{\delta(z)}
\qquad(v\in T^{1,0}_zD).\label{eq:5.6}
\end{equation}
Let \(\gamma\) be a piecewise \(C^1\) curve in the collar $V_{\epsilon}$. The differential
of the distance function is controlled by the complex-normal component of
the velocity. Consequently, with a constant \(b_2>0\) accounting for the
convention used to identify real and \((1,0)\) tangent vectors, it follows from
\eqref{eq:5.6} that
\[
\begin{aligned}
L_h(\gamma|_{[s,t]})
&\ge b_2\int_s^t\frac{|(\delta\circ\gamma)'(r)|}
{\delta(\gamma(r))}\,\dd r\\
&\ge b_2\left|\log\frac{\delta(\gamma(t))}{\delta(\gamma(s))}\right|.
\end{aligned}
\]
Hence every curve tending to \(\partial D\) has infinite \(h\)-length.
Patching \(h\) with an arbitrary smooth positive definite Hermitian metric in
a compact interior region, we obtain a complete metric in \(D\). Therefore,
\(h\) is complete at infinity.
\end{proof}

\begin{proof}[Proof of Theorem~\ref{thm:boundary}]
The implications \textup{(i)} \(\Rightarrow\) \textup{(ii)} and
\textup{(iii)} \(\Rightarrow\) \textup{(i)} are immediate.
If \textup{(iv)} holds, let \(P:\B^n\to D\) be a holomorphic
covering map. Since
\[
P^*\kappa_D=\kappa_{\B^n},
\]
and \(\kappa_{\B^n}\) is Hermitian K\"ahler with holomorphic
sectional curvature \(-4\), condition \textup{(iii)} follows.

It remains to prove \textup{(ii)} \(\Rightarrow\) \textup{(iv)}.
From Lemma~\ref{lem:retract}, Theorem~\ref{theo:cone} and the preceding
construction, we obtain a Hermitian K\"ahler metric \(h\) on a boundary collar
satisfying \(K_h\equiv-4\). By Lemma~\ref{lem:infinity}, \(h\) is complete
at infinity.

The domain $D$ is Stein. Since $n\ge2$, it has one end
(\cite{GZ}). We may therefore choose a compact set
$K\Subset D$ such that $D\setminus K$ is connected and contained
in the collar where $h$ is defined. The metric $h$ is K\"ahler
with constant holomorphic sectional curvature $-4$, so it is
Hermitian locally symmetric. It follows from Lemma~\ref{lem:infinity} that $h$
is complete at infinity. Thus all the hypotheses of Theorem 1.1 in 
\cite{GZ} are satisfied. The metric $h|_{D\setminus K}$
extends to a complete Hermitian locally symmetric metric
$\widetilde h$ on $D$, which is K\"ahler (\cite{GZ}). Moreover,
\[
K_{\widetilde h}\equiv-4\quad\text{on }D\setminus K.
\]
Hence, applying  Corollary~1.2 in   \cite{GZ}, we have \(\widehat D\simeq\B^n\).
Therefore, the four conditions are equivalent:
\[
\text{(i)}\Rightarrow\text{(ii)}\Rightarrow\text{(iv)}
\Rightarrow\text{(iii)}\Rightarrow\text{(i)}.
\]
If \(D\) is simply connected, then the universal covering map is a
biholomorphism, and hence \(D\simeq\B^n\).
\end{proof}

\section{Examples of smooth weakly K\"ahler--Finsler metrics}
In this section, we give two examples of smooth weakly K\"ahler--Finsler metrics that are not
K\"ahler--Finsler.

\begin{proposition}\label{prop:balanced}
The domain
\[
D=\{(z,w)\in\C^2:
|z|^2+|w|^2+|z|^2|w|^2<1\}
\]
has smooth strongly convex boundary. Its Kobayashi metric is smooth,
strongly pseudoconvex and weakly K\"ahler--Finsler, but is not
K\"ahler--Finsler.
\end{proposition}

\begin{proof}
Set \(f(z,w)=|z|^2+|w|^2+|z|^2|w|^2\).
A direct calculation of its real Hessian gives
\[
\nabla^2f\ge (2-4|z||w|)I.
\]
Thus \(f\) is strictly convex on the Euclidean unit ball, which
contains \(D\), and hence \(D\) is convex. On \(\overline D\),
\[
2|z||w|+|z|^2|w|^2\le1,
\]
so \(|z||w|\le\sqrt2-1\) and
\[
\nabla^2f\ge(6-4\sqrt2)I>0.
\]
Moreover, \(df\ne0\) on \(f=1\), since the derivative of \(f\) in
the radial direction is positive there. Therefore \(D\) has smooth
strongly convex boundary.

Since \(D\) is balanced and convex, \(\kappa_D(0;\cdot)\) equals
its Minkowski functional. Solving the boundary equation, we get
\[
Q(\alpha,\beta):=\kappa_D(0;\alpha,\beta)^2
=\frac{s+\sqrt{s^2+4|\alpha|^2|\beta|^2}}2,
\qquad s=|\alpha|^2+|\beta|^2.
\]
For the standard basis vectors \(e_1,e_2\), we have
\[
Q(e_1)=Q(e_2)=1,\qquad
Q(e_1+e_2)=Q(e_1-e_2)=1+\sqrt2.
\]
Consequently,
\[
Q(e_1+e_2)+Q(e_1-e_2)
\ne 2Q(e_1)+2Q(e_2).
\]
Thus \(Q\) is not a Hermitian quadratic form, and \(\kappa_D\)
is not Hermitian. It follows from Theorem~\ref{thm:convex} that
\(\kappa_D\) is smooth, strongly pseudoconvex and weakly
K\"ahler--Finsler, but is not K\"ahler--Finsler.
\end{proof}

\begin{proposition}\label{prop:shear}
The domain
\[
D_{\mathrm{sh}}
=\{(z,w)\in\C^2:
|z-2w^2|^2+|w|^2+|z-2w^2|^2|w|^2<1\}
\]
is bounded and nonconvex, with real-analytic strongly pseudoconvex
boundary. Its Kobayashi metric is smooth, strongly pseudoconvex
and weakly K\"ahler--Finsler, but is not K\"ahler--Finsler.
\end{proposition}

\begin{proof}
Let \(D\) be the domain in Proposition~\ref{prop:balanced}, and set
\[
f(\zeta,\eta)=|\zeta|^2+|\eta|^2+|\zeta|^2|\eta|^2.
\]
The polynomial automorphism
\[
\Phi(\zeta,\eta)=(\zeta+2\eta^2,\eta),
\qquad
\Psi=\Phi^{-1},\quad
\Psi(z,w)=(z-2w^2,w),
\]
maps \(D\) onto \(D_{\mathrm{sh}}\).
It follows that \(D_{\mathrm{sh}}\) is bounded and has
real-analytic strongly pseudoconvex boundary.

To see that it is not convex, take
\[
p_\pm=\left(\frac65,\pm\frac12\right).
\]
Then
\[
f(\Psi(p_\pm))=\frac{69}{80}<1,
\qquad
f\left(\Psi\left(\frac{p_++p_-}{2}\right)\right)
=\frac{36}{25}>1.
\]
Thus \(p_\pm\in D_{\mathrm{sh}}\), whereas their midpoint is
outside \(D_{\mathrm{sh}}\).

Using biholomorphic invariance, we have
\[
\kappa_{D_{\mathrm{sh}}}=\Psi^*\kappa_D.
\]
Smoothness, strong pseudoconvexity, and the weakly
K\"ahler--Finsler and K\"ahler--Finsler conditions are preserved
under biholomorphisms. The assertions therefore follow from
Proposition~\ref{prop:balanced}.
\end{proof}

\begin{remark}
For \(t>0\), Ge \cite{Ge} considered the bounded strongly convex Reinhardt
domain
\[
\Omega_t=\{(z_1,z_2)\in\C^2:|z_1|^2+|z_2|^2+t|z_1|^4<1\}.
\]
Its Kobayashi metric is smooth, strongly pseudoconvex and weakly
K\"ahler--Finsler, but is not K\"ahler--Finsler. Ge also obtained the
explicit formula
\[
\kappa_{\Omega_t}(0;v)^2
=\frac{\|v\|^2+\sqrt{\|v\|^4+4t|v_1|^4}}2,\qquad v=(v_1,v_2)\in\C^2.
\]
This provides another example of a smooth weakly K\"ahler--Finsler metric
that is not K\"ahler--Finsler, alongside the examples in this section.
\end{remark}

\bigskip
\noindent\textbf{Jinling Li:} School of Science, Jimei University,
Xiamen 361021, China\\
\textit{Email address:} \href{mailto:lijl@jmu.edu.cn}{lijl@jmu.edu.cn}

\medskip
\noindent\textbf{Longjie Li:} School of Mathematical Sciences,
Xiamen University, Xiamen 361005, China\\
\textit{Email address:} \href{mailto:1853248227@qq.com}{1853248227@qq.com}

\medskip
\noindent\textbf{Chunhui Qiu:} School of Mathematical Sciences,
Xiamen University, Xiamen 361005, China\\
\textit{Email address:} \href{mailto:chqiu@xmu.edu.cn}{chqiu@xmu.edu.cn}

\end{document}